\documentclass[12pt]{article}
\usepackage{amsmath}
\usepackage{mathtools}
\usepackage{amsfonts}
\usepackage{amsthm}
\usepackage[english]{babel} 
\usepackage[iso-8859-7]{inputenc}
\usepackage{bold-extra}
\usepackage{indentfirst}
\newtheorem{definition}{Definition}
\newtheorem{lemma}[definition]{Lemma}
\newtheorem{theorem}[definition]{Theorem}
\newtheoremstyle{r}
{}
{}
{\normalfont}
{}
{\scshape}
{.}
{ }
{}
\theoremstyle{r}

\emergencystretch=\maxdimen
\title{An extension of the universal power series of Seleznev}
\author{K. Maronikolakis and V. Nestoridis}
\date{}
\begin{document}
	\latintext
	\maketitle
	\begin{abstract}
		We show generic existence of power series $a = \sum\limits_{n = 0}^{\infty}a_nz^n, a_n \in \mathbb{C}$, such that the sequence $T_N(a)(z) = \sum\limits_{n=0}^{N}b_n(a_0, \dots , a_n)z^n,\linebreak N = 0, 1, 2 \dots$ approximates every polynomial uniformly on every compact set $K \subset \mathbb{C}\setminus\{0\}$ with connected complement. The functions $b_n : \mathbb{C}^{n + 1}\rightarrow\mathbb{C}$ are assumed to be continuous and such that for every $a_0, a_1, \dots , a_{n - 1} \in \mathbb{C}$, the function $\mathbb{C}\ni z \rightarrow b_n(a_0, a_1, \dots , a_{n - 1}, z)$ is onto $\mathbb{C}$. This clearly covers the case of linear functions $b_n$: $b_n(a_0, \dots, a_n) = \sum\limits_{k = 0}^{n}\lambda_{n,k}a_k, \lambda_{n,k} \in \mathbb{C}, \lambda_{n,n} \neq 0$.
	\end{abstract}
	\textbf{AMS classification number:} 30K05.\\
	\textbf{Key words and phrases:} Universal Taylor series, Baire's theorem, generic property, Mergelyan's theorem.
	\newpage
	\section{Introduction}
	A classical result of Seleznev states that there exists a formal power series $\sum\limits_{n = 0}^{\infty}a_nz^n, a_n \in \mathbb{C}$, such that its partial sums $S_N(z) = \sum\limits_{n = 0}^{N}a_nz^n$ have the following universal approximation property:
	
	For every compact set $K \subset \mathbb{C}\setminus\{0\}$ with connected complement and for every function $h : K\rightarrow\mathbb{C}$, which is continuous on $K$ and holomorphic in the interior of $K$, there exists a strictly increasing sequence of natural numbers $(\lambda_m)_{m \in \mathbb{N}}$, such that  $(S_{\lambda_m}(z))_{m \in \mathbb{N}}$ converges to $h$ uniformly on $K$, as $m \longrightarrow +\infty$ \cite{Bayart}, \cite{Seleznev}.
	
	If we identify the formal power series $\sum\limits_{n = 0}^{\infty}a_nz^n$ with the sequence $a = (a_0, a_1, \dots)$, then the previous fact holds on a $G_\delta$ and dense subset of $\mathbb{C}^{\aleph_0}$ endowed with the product topology \cite{Bayart}.
	
	It can easily be seen that the previous power series have zero radius of convergence. For universal Taylor series with strictly positive radius of convergence we refer to \cite{Chui}, \cite{Luh} and \cite{Nestoridis}; see also \cite{Bayart}.
	
	In this paper we extend the result of Seleznev in the case where the universal approximation is not achieved by $S_N(z) = \sum\limits_{n = 0}^{N}a_nz^n$, but it is achieved by $T_N(a)(z) = \sum\limits_{n=0}^{N}b_n(a_0, \dots , a_n)z^n$, where $b_n : \mathbb{C}^{n + 1}\rightarrow\mathbb{C}$ are given functions. Our assumptions are the continuity of such $b_n$ and that for every fixed $a_0, a_1, \dots , a_{n - 1} \in \mathbb{C}$, the function $\mathbb{C}\ni z \rightarrow b_n(a_0, a_1, \dots , a_{n - 1}, z)$ is onto $\mathbb{C}$. In particular, our results are valid if the functions $b_n$ are linear: $b_n(a_0, \dots, a_n) = \sum\limits_{k = 0}^{n}\lambda_{n,k}a_k, \lambda_{n,k} \in \mathbb{C}, \lambda_{n,n} \neq 0$.
	
	In this case, we prove that the universal approximation property is generic topologically and algebraically. That is, the set $U$ of universal power series $a \in \mathbb{C}^{\aleph _0}$ is a $G_\delta$ and dense subset of $\mathbb{C}^{\aleph _0}$ (topological genericity) and it contains a dense vector subspace except 0 (algebraic genericity).
	
	We also notice that our results easily imply the fact that for the generic power series $a = \sum\limits_{n = 0}^{\infty}a_nz^n$, the power series $\sum\limits_{n = 0}^{\infty}a_nz^n$, $\sum\limits_{n = 0}^{\infty}\frac{a_0 + \dots + a_n}{n + 1}z^n$ have zero radius of convergence.
	\section{Main Result}
	\begin{definition}
		For every integer $n\geq0$ let $b_n : \mathbb{C}^{n + 1}\rightarrow\mathbb{C}$ be a continuous function such that for every $a_0, a_1, \dots , a_{n - 1} \in \mathbb{C}$, the function $$\mathbb{C}\ni z \rightarrow b_n(a_0, a_1, \dots , a_{n - 1}, z)$$ is onto $\mathbb{C}$. Let $a = (a_0, a_1, \dots) \in \mathbb{C}^{\aleph_0}$. For every integer $N\geq0$ and $z \in \mathbb{C}$ we set $T_N(a)(z) = \sum\limits_{n=0}^{N}b_n(a_0, \dots , a_n)z^n$. Let $\mu$ be an infinite subset of $\mathbb{N}$.
		We define $U^\mu$ to be the set of $a \in \mathbb{C}^{\aleph_0}$, such that for every compact set $K \subset \mathbb{C}\setminus\{0\}$ with connected complement and for every function $h : K\rightarrow\mathbb{C}$, which is continuous on $K$ and holomorphic in the interior of $K$, there exists a strictly increasing sequence of integers $(\lambda_m)_{m \in \mathbb{N}},\lambda_m \in \mu$ such that  $(T_{\lambda_m}(a)(z))_{m \in \mathbb{N}}$ converges to $h(z)$ uniformly on $K$, as $m \longrightarrow +\infty$.
	\end{definition}
	We notice that if we assume that there exists a sequence of integers $(\lambda_m)_{m \in \mathbb{N}},\lambda_m \in \mu$, not necessarily strictly increasing, such that  $(T_{\lambda_m}(a)(z))_{m \in \mathbb{N}}$ converges to $h(z)$ uniformly on $K$ then the two definitions are equivalent; see \cite{Vlachou}.
	
	Considering the set $U^\mu$  as a subset of the space $\mathbb{C}^{\aleph_0}$ endowed with the product topology, we shall prove that $U^\mu$ is a countable intersection of open dense sets. Since $\mathbb{C}^{\aleph_0}$ is a metrizable complete space, Baire's theorem is at our disposal and so $U^\mu$ is a dense $G_\delta$ set.
	
	The following lemma is well known \cite{Bayart}, \cite{Luh}:
	\begin{lemma}
		There exists a sequence of infinite compact sets $K_m \subset \mathbb{C}\setminus\{0\},\\ m = 1, 2, \dots$ with connected complements, such that the following holds: every non-empty compact set  $K \subset \mathbb{C}\setminus\{0\}$ having connected complement is contained in some $K_m$.
	\end{lemma}
	We fix now a sequence $K_m, m = 1, 2, \dots$ as in Lemma 2. Let $f_j, j = 1, 2, \dots$ be an enumeration of all polynomials having coefficients with rational coordinates. For any integers $m, j, s, N$ with $m \geq 1, j \geq 1, s \geq 1, N \geq 0$, we denote by $E(m, j, s, N)$ the set $$E(m, j, s, N) := \Big\{a \in \mathbb{C}^{\aleph_0} : \sup_{z \in K_m}\big|T_N(a)(z) - f_j(z)\big| < \frac{1}{s}\Big\}.$$
	\begin{lemma}
		$U^\mu$ can be written as follows:$$U^\mu = \bigcap_{m=1}^{\infty}\bigcap_{j=1}^{\infty}\bigcap_{s=1}^{\infty}\bigcup_{N \in \mu}E(m, j, s, N).$$
	\end{lemma}
	\begin{proof}
		The inclusion $U^\mu \subseteq \bigcap\limits_{m=1}^{\infty}\bigcap\limits_{j=1}^{\infty}\bigcap\limits_{s=1}^{\infty}\bigcup\limits_{N \in \mu}E(m, j, s, N)$ follows obviously from the definitions of $U^\mu$ and $E(m, j, s, N).$ Let $$a \in  \bigcap_{m=1}^{\infty}\bigcap_{j=1}^{\infty}\bigcap_{s=1}^{\infty}\bigcup_{N \in \mu}E(m, j, s, N).$$ We shall show that $a \in U^\mu$. Let $K \subset \mathbb{C}\setminus\{0\}$ be a non-empty compact set having connected complement and $h : K \rightarrow \mathbb{C}$ a function, which is continuous on $K$ and holomorphic in the interior of $K$. Let $\varepsilon > 0$. We have to determine an integer $N \in \mu$, such that $$\sup_{z \in K}\big|T_N(a)(z) - h(z)\big| < \varepsilon.$$
		By Mergelyan's theorem there exists a polynomial $f_j, j = 1, 2, \dots$ having coefficients whose coordinates are both rational, such that $$\sup_{z \in K}\big|h(z) - f_j(z)\big| < \frac{\varepsilon}{2}.$$ There exists a compact set with connected complement $K_m, m = 1, 2, \dots$ given by Lemma 2, such that $K \subseteq K_m$. We can determine an $s$, such that $\frac{1}{s} < \frac{\varepsilon}{2}$. Then we have $a \in \bigcup\limits_{N \in \mu}E(m, j, s, N)$. Thus, there exists an integer $N \in \mu$, such that $$\sup_{z \in K_m}\big|T_N(a)(z) - f_j(z)\big| < \frac{1}{s}.$$ As we have $\sup\limits_{z \in K}\big|h(z) - f_j(z)\big| < \frac{\varepsilon}{2}$, $\sup\limits_{z \in K_m}\big|T_N(a)(z) - f_j(z)\big| < \frac{1}{s} < \frac{\varepsilon}{2}$ and $K \subseteq K_m$, the triangular inequality implies $$\sup_{z \in K}\big|T_N(a)(z) - h(z)\big| < \varepsilon.$$ This proves that $a \in U^\mu$ and completes the proof.
	\end{proof}
	\begin{lemma}
		For every integer $m \geq 1, j \geq 1, s \geq 1$ and $N \in \mu$, the set $E(m, j, s, N)$ is open in the space $\mathbb{C}^{\aleph_0}$.
	\end{lemma}
	\begin{proof}
		Let $a = (a_0, a_1, \dots) \in E(m, j, s, N)$. Then we have $$\sup_{z \in K_m}\big|T_N(a)(z) - f_j(z)\big| < \frac{1}{s}.$$ Let $M := \max\big\{1, \sup_{z \in K_m}|z|^N\big\}$. We set now: $$\varepsilon = \dfrac{\frac{1}{s} - \sup_{w \in K_m}\big|T_N(a)(w) - f_j(w)\big|}{2(N + 1)M} > 0.$$
		For $n = 0, 1, \dots, N$ the function $b_n$ is continuous at $(a_0, a_1, \dots, a_n)$, so there exists $\delta_n > 0$ such that $|b_n(c_0, c_1, \dots, c_n) - b_n(a_0, a_1, \dots, a_n)| < \varepsilon$ for $(c_0, c_1, \dots, c_n) \in \mathbb{C}^{n + 1}$ with $\sqrt{\sum\limits_{k=0}^{n}|c_k - a_k|^2} < \delta_n$. We set $\delta = \min\{\delta_0, \delta_1, \dots, \delta_N\}$. Suppose that $c = (c_0, c_1, \dots) \in \mathbb{C}^{\aleph_0}$ satisfies $|c_k - a_k| < \frac{\delta}{\sqrt{N + 1}}$ for $k = 0, 1, \dots,N$. We shall show that $$\sup_{z \in K_m}\big|T_N(c)(z) - f_j(z)\big| < \frac{1}{s}$$ and therefore that $c \in E(m, j, s, N)$. This will prove that $E(m, j, s, N)$ is indeed open. For $n = 0, 1, \dots, N$ we have $$\sqrt{\sum\limits_{k=0}^{n}|c_k - a_k|^2} < \sqrt{\sum\limits_{k=0}^{n}\Big(\frac{\delta}{\sqrt{N + 1}}\Big)^2} \leq \sqrt{\sum\limits_{k=0}^{N}\frac{\delta^2}{N + 1}} = \delta \leq \delta_n$$ and so $|b_n(c_0, c_1, \dots, c_n) - b_n(a_0, a_1, \dots, a_n)| < \varepsilon$. For $z \in K_m$, we have $$\big|T_N(c)(z) - f_j(z)\big| \leq \big|T_N(c)(z) - T_N(a)(z)\big| + \big|T_N(a)(z) - f_j(z)\big| =$$\\$$= \big|\sum_{n=0}^{N}b_n(c_0, c_1, \dots, c_n)z^n - \sum_{n=0}^{N}b_n(a_0, a_1, \dots, a_n)z^n\big| + \big|T_N(a)(z) - f_j(z)\big| \leq$$\\$$\leq \sum_{n=0}^{N}|b_n(c_0, c_1, \dots, c_n) - b_n(a_0, a_1, \dots, a_n)|\cdot|z|^n + \big|T_N(a)(z) - f_j(z)\big| <$$\\
		$$< \sum_{n=0}^{N}\varepsilon M + \big|T_N(a)(z) - f_j(z)\big| =$$\\$$= \sum_{n=0}^{N}\dfrac{\frac{1}{s} - \sup_{w \in K_m}\big|T_N(a)(w) - f_j(w)\big|}{2(N + 1)} + \big|T_N(a)(z) - f_j(z)\big| =$$\\$$= \frac{1}{2s} - \frac{1}{2}\sup_{w \in K_m}\big|T_N(a)(w) - f_j(w)\big| + \big|T_N(a)(z) - f_j(z)\big|.$$ Hence, $$\sup_{z \in K_m}\big|T_N(c)(z) - f_j(z)\big| \leq \frac{1}{2s} < \frac{1}{s}$$
		and the proof is completed.
	\end{proof}
	\begin{lemma}
		For every integer $m \geq 1, j \geq 1$ and $s \geq 1$, the set $\bigcup\limits_{N \in \mu}E(m, j, s, N)$ is open and dense in the space $\mathbb{C}^{\aleph_0}$.
	\end{lemma}
	\begin{proof}
		By Lemma 4 the sets $E(m, j, s, N), N \in \mu$ are open. Therefore the same is true for the union $\bigcup\limits_{N \in \mu}E(m, j, s, N)$. We shall prove that this set is also dense. Let $a = (a_0, a_1, \dots) \in \mathbb{C}^{\aleph_0}, N_0$ be an integer such that $N_0 \geq 0$ and $\varepsilon > 0$. It suffices to find $N \in \mu$ and $c = (c_0, c_1, \dots) \in E(m, j, s, N)$, such that $$|c_n - a_n| < \varepsilon \text{ for } n \leq N_0.$$ Let $M := \sup\limits_{z \in K_m}|z|^{N_0 + 1}$. We set $c_n = a_n$ for $n \leq N_0$ and so $ b_n(c_0, c_1, \dots, c_n) =$\\$= b_n(a_0, a_1, \dots, a_n) \text{ for } n \leq N_0.$ We need to find $N \in \mu$ such that $$\sup_{z \in K_m}\big|T_N(c)(z) - f_j(z)\big| < \frac{1}{s}.$$ We have $$\sup\limits_{z \in K_m}\big|T_N(c)(z) - f_j(z)\big| = \sup\limits_{z \in K_m}\big|\sum_{n=0}^{N}b_n(c_0, c_1, \dots, c_n)z^n - f_j(z)\big| =$$\\
		$$= \sup\limits_{z \in K_m}\big|\sum_{n=N_0 + 1}^{N}b_n(c_0, c_1, \dots, c_n)z^n + \sum_{n=0}^{N_0}b_n(a_0, a_1, \dots, a_n)z^n - f_j(z)\big| =$$\\
		$$= \sup\limits_{z \in K_m}\big|z^{N_0 + 1}\sum_{n=N_0 + 1}^{N}b_n(c_0, c_1, \dots, c_n)z^{n - N_0 - 1} + \sum_{n=0}^{N_0}b_n(a_0, a_1, \dots, a_n)z^n - f_j(z)\big| =$$\\
		$$= \sup\limits_{z \in K_m}|z^{N_0 + 1}|\cdot\big|\sum_{n=N_0 + 1}^{N}b_n(c_0, \dots, c_n)z^{n - N_0 - 1} - \dfrac{f_j(z) - \sum_{n=0}^{N_0}b_n(a_0, \dots, a_n)z^n}{z^{N_0 + 1}}\big| \leq$$\\
		$$\leq M\sup\limits_{z \in K_m}\big|\sum_{n=N_0 + 1}^{N}b_n(c_0, c_1, \dots, c_n)z^{n - N_0 - 1} - \dfrac{f_j(z) - \sum_{n=0}^{N_0}b_n(a_0, a_1, \dots, a_n)z^n}{z^{N_0 + 1}}\big|.$$
		Since $0 \notin K$ and $K^c$ is connected, by Mergelyan's theorem there exists a polynomial $p(z) = p_0 + p_1z + \dots + p_mz^m$ such that $$\sup\limits_{z \in K_m}\big|p(z) - \dfrac{f_j(z) - \sum_{n=0}^{N_0}b_n(a_0, a_1, \dots, a_n)z^n}{z^{N_0 + 1}}\big| < \frac{1}{2Ms}.$$ The function $$\mathbb{C}\ni z \rightarrow b_{N_0 + 1}(a_0, a_1, \dots , a_{N_0}, z)$$ is onto $\mathbb{C}$ so the equation $ b_{N_0 + 1}(a_0, a_1, \dots , a_{N_0}, z) = p_0$ has a solution $c_{N_0 + 1} \in \mathbb{C}$. Similarly, we can find $c_{N_0 + 2}, \dots , c_{N_0 + m + 1}$ such that $b_{N_0 + n + 1}(c_0, c_1, \dots , c_{N_0 + n + 1}) = p_n$ for $n = 1, 2, \dots, m$ and $c_{N_0 + m +2}, c_{N_0 + m +3}, \dots$ such that $b_{N_0 + n + 1}(c_0, c_1, \dots , c_{N_0 + n + 1}) = 0$ for $n > m$. By choosing $N \in \mu$ such that $N \geq m + N_0 + 1$ we have $$\sup_{z \in K_m}\big|T_N(c)(z) - f_j(z)\big| \leq \frac{1}{2s} < \frac{1}{s}.$$ This proves that the set $\bigcup\limits_{N \in \mu}E(m, j, s, N)$ is indeed dense.
	\end{proof}
	\begin{theorem}
		Under the above assumptions and notation, the set $U^\mu$ is a $G_\delta$ and dense subset of the space $\mathbb{C}^{\aleph_0}$.
	\end{theorem}
	\begin{proof}
		The result is obvious by combining the previous lemmas with Baire's Theorem.
	\end{proof}
	\begin{theorem}
		Under the above assumptions and notation, assuming in addition that the functions $b_n$ are linear, then the set $U^\mu \cup \{0\}$ contains a vector space, dense in $\mathbb{C}^{\aleph_0}$.
	\end{theorem}
	The proof uses the result of Theorem 6, follows the lines of the implication $(3) \implies (4)$ of the proof of Theorem 3 in \cite{Bayart} and is omitted.
	\section{Remarks and Comments}
	\par
	The assumptions of the previous section are valid in particular when $b_n(a_0, \dots, a_n) = a_n$ which gives the classical result of Seleznev. Also, it covers the interesting case $b_n(a_0, \dots, a_n) = \frac{a_0 + \dots + a_n}{n + 1}$.
	
	More generally, if $\psi_n:\mathbb{C}\rightarrow\mathbb{C}, n= 0, 1, \dots$ are homeomorphisms and $\lambda_{n,k} \in \mathbb{C}$, $ 0 \leq k \leq n$, $ n \in \mathbb{N}$, $ \lambda_{n,n} \neq 0$, we can set $b_n(a_0, \dots, a_n) = \psi_n\big(\sum_{k = 0}^{n}\lambda_{n,k}a_k\big)$ and the results of the previous section are valid.
	
	Another remark is that in order to prove that $U^\mu$ is a $G_\delta$ set we only need the continuity of the functions $b_n : \mathbb{C}^{n + 1}\rightarrow\mathbb{C}$ (1). We do not need the assumption that for every $a_0, a_1, \dots , a_{n - 1} \in \mathbb{C}$, the function $\mathbb{C}\ni z \rightarrow b_n(a_0, a_1, \dots , a_{n - 1}, z)$ is onto $\mathbb{C}$ (2). It is also true that using only assumption (2) we can prove that $U^\mu$ is dense in $\mathbb{C}^{\aleph_0}$.
	
	Indeed, from the classical result of Seleznev, there exist formal power series $c_0 + c_1z + c_2z^2 + \dots$ such  that for every compact set $K \subset \mathbb{C}\setminus\{0\}$ with connected complement and for every function $h : K\rightarrow\mathbb{C}$, which is continuous on $K$ and holomorphic in the interior of $K$, there exists a sequence of integers $(\lambda_m)_{m \in \mathbb{N}},\lambda_m \in \mu$ such that $c_0 + c_1z + \dots + c_{\lambda_m}z^{\lambda_m} \longrightarrow h$ uniformly on K, as $m \longrightarrow +\infty$. Also, we can modify a finite set of coefficients $c_k$ and still have the same universal approximation.
	
	Let $a_0, a_1 , \dots, a_{N_0} \in \mathbb{C}$ be fixed. It suffices to show that we can find $a_{N_0 + 1}, a_{N_0 + 2}, \dots \in \mathbb{C}$ such that $a = (a_0, a_1 , \dots, a_{N_0}, a_{N_0 + 1}, a_{N_0 + 2}, \dots) \in U^\mu$. We set $\delta_k = b_k(a_0, a_1 , \dots, a_{k}), 0 \leq k \leq N_0$. As we have already mentioned, we can find a formal power series of Seleznev $c = (c_0, c_1, \dots)$ satisfying $c_k = \delta_k, 0 \leq k \leq N_0$.
	Then, because the function $\mathbb{C}\ni z \rightarrow b_{N_0 + 1}(a_0, a_1, \dots , a_{N_0}, z)$ is onto $\mathbb{C}$, we can find $a_{N_0 + 1}$ such that $b_{N_0 + 1}(a_0, a_1, \dots , a_{N_0}, a_{N_0 + 1}) = c_{N_0 + 1}$. Continuing in this way we can find $a = (a_0, a_1 , \dots, a_{N_0}, a_{N_0 + 1}, a_{N_0 + 2}, \dots) \in \mathbb{C}^{\aleph_0}$ such that $b_n(a_0, a_1, \dots , a_n) = c_n$ for every $n \in \mathbb{N}$. Therefore $a \in U^\mu$. This proves that $U^\mu$ is dense.
	\\[2\baselineskip]
	\par
	\textit{Acknowledgment}. We would like to thank G. Costakis for helpful communications.
	
	\itshape
	Department of Mathematics\\
	Panepistimiopolis\\
	National and Kapodistrian University of Athens\\
	Athens, 15784\\
	Greece\\[\baselineskip]
	E-mail Addresses:\\
	conmaron@gmail.com\\
	vnestor@math.uoa.gr
\end{document}